\documentclass[12pt]{article}
\usepackage{amsmath}
\usepackage{amssymb}
\usepackage{nccmath}
\usepackage{amsthm}
\usepackage{graphicx}
\usepackage{enumerate}
\usepackage{csquotes}
\usepackage{bbold}
\usepackage{amsmath}
\usepackage{amssymb}
\usepackage{nccmath}
\usepackage{amsthm}
\usepackage{url} 

\usepackage{csquotes}
\MakeOuterQuote{"}

\usepackage{natbib}
\allowdisplaybreaks

\newcommand{\blind}{1}

\theoremstyle{plain} 
\newtheorem{theorem}{Theorem}
\newtheorem{lemma}{Lemma}
\newtheorem{definition}{Definition}[section]

\begin{document}

\def\spacingset#1{\renewcommand{\baselinestretch}%
{#1}\small\normalsize} \spacingset{1}


\if1\blind
{
  \title{\bf The Effect of Sample Size and Missingness on Inference with Missing Data}
  \author{Julian Morimoto \thanks{
    The author thanks Professor Kosuke Imai for his feedback on this project. \textit{This project has not been funded.}}\hspace{.2cm}}
  \maketitle
} \fi

\if0\blind
{
  \bigskip
  \bigskip
  \bigskip
  \begin{center}
    {\LARGE\bf The Effect of Sample Size and Missingness on Inference with Missing Data
}
\end{center}
  \medskip
} \fi

\bigskip

\begin{abstract}
When are inferences (whether Direct-Likelihood, Bayesian, or Frequentist) obtained from partial data valid? This paper answers this question by offering a new asymptotic theory about inference with missing data that is more general than existing theories. It proves that as the sample size increases and the extent of missingness decreases, the average-loglikelihood function generated by partial data and that \textit{ignores} the missingness mechanism will converge in probability to that which would have been generated by complete data; and if the data are Missing at Random, this convergence depends only on sample size. Thus, inferences from partial data, such as posterior modes, confidence intervals, likelihood ratios, test statistics, and indeed, \textit{all} quantities or features derived from the partial-data loglikelihood function, will be consistently estimated. Additionally, the missing data mechanism has asymptotically no effect on parameter estimation and hypothesis testing if the data are Missing at Random. This adds to previous research which has only proved the consistency and asymptotic normality of the posterior mode. Practical implications are discussed, and the theory is illustrated through simulation using a previous study of International Human Rights Law. 
\end{abstract}


\section{Introduction}
Missing data are a big problem in scientific research, and it is not always clear when inferences drawn from incomplete datasets will be similar to those that would have been drawn from complete ones. In the social sciences, on average, about 50\% of survey respondents will fail to provide complete answers \citep{king_analyzing_2001}. Missing phenotype data is also often an inevitable and major hurdle in genetic experiments where quantitative traits are not observed or certain expressions are not realized \citep{nielsen_non-negative_2014, bobb_multiple_2011}. In astronomy, unfavorable weather, man-made radio emissions, and emissions from the sun can cause missing-data \citep{george_effect_2015}. Consequently, researchers have tried to find ways to draw valid inferences from partially complete datasets. One technique is to draw inferences from the likelihood function generated by partial data \citep{little_statistical_2019, little_conditions_2017}, which has proven quite successful in many contexts \citep{murray_multiple_2018, shin_maximum_2017, von_hippel_new_2016, honaker_amelia_2011, nielsen_proper_2003}. This approach involves approximating the loglikelihood function generated by the partially observed dataset by "integrating" over the missing values (this process is shown in more detail in the proofs of Theorems \ref{thm:thm1} and \ref{thm:thm2}). One can then use the resulting function to draw Direct-Likelihood, Bayesian, or Frequentist inferences from the partial data \citep{seaman_what_2013}. 

It is important to note that this function \textit{ignores} the missingness mechanism \citep{little_statistical_2019}. As a terminological point, whenever this paper refers to average-loglikelihood/loglikelihood/likelihood functions generated by partial data, this paper means those functions that focus only on the distribution of the underlying data (whether observed or missing) and that do not take into account the missingness mechanism. Focusing on these functions is consistent with the approach of other scholars who have done so given that it is difficult, if not impossible, to know anything about the missingness mechanism in practice \citep{little_statistical_2019, little_conditions_2017}.

Despite the popularity of the partial-likelihood approach, much is still unknown about when this method yields asymptotically valid results. \cite{little_statistical_2019} note that:

\begin{quote}
    In one high-level sense, there is no formal difference between [likelihood] or Bayes inferences for incomplete data and [likelihood] or Bayes inference for complete data?
    Asymptotic standard errors obtained from the information matrix are somewhat more questionable with missing data, however, because the observed data do not generally constitute an iid sample, and simple results that imply the large sample normality of the likelihood function do not immediately apply. Other complications arise from dealing with the process that creates missing data.
\end{quote}

Thus, it is not immediately obvious whether (and which) inferences drawn from the partial likelihood function will be asymptotically valid. Current asymptotic theories are limited in that they have only proved the consistency and asymptotic normality of the partial-likelihood function?s posterior mode when the data are Missing at Random \citep{little_statistical_2019, little_conditions_2017, seaman_what_2013, takai_asymptotic_2013, nielsen_inference_1997}. This leaves many important unanswered questions. Will other quantities derived from the loglikelihood such as uncertainty estimates, confidence intervals, likelihood ratios, and test statistics be consistently estimated? What impact does the extent of missingness have, if any, on inference beyond consistency and asymptotic normality, and why? How does sample size interact with missingness? Will inferences from Missing not at Random data be analogously valid under certain conditions? When are inferences about subvectors of parameters consistent \citep{little_conditions_2017}? What effect does the missingness mechanism have on parameter estimation and hypothesis testing? Much is still unknown about whether and when inferences obtained from partial data will be asymptotically valid.

This paper fills in this knowledge gap by offering a new asymptotic theory about how sample size and the missingness mechanism affect the validity of inferences from missing data. It proves two Theorems that explain the effect of sample size and missingness on inference with missing data. Theorem \ref{thm:thm1} proves that the average-loglikelihood function (ignoring the missingness mechanism) generated by partial data will converge in probability to the function that would have been generated by complete data as the sample size increases and the extent of missingness decreases. Theorem \ref{thm:thm2} proves that for data that are Missing at Random (technically "(Almost) Everywhere Missing at Random", which is different from "Realized Missing at Random" \citep{seaman_what_2013}, but this point will be further discussed below in Theorem \ref{thm:thm2}), this convergence depends only on sample size. This suggests that inferences obtained from partial datasets will approximate those that would have been obtained from complete ones. 

These two theorems also have three important practical implications. First, inferences drawn from partial datasets will be close to their true values given a sufficiently large sample size and a sufficiently small amount of missing data. Second, when the data are Missing at Random, this approximation depends only on the sample size and not the proportion of missing data. This means that a researcher could have a dataset with a large amount of missing data and still be able to draw valid inferences if that dataset is sufficiently large. The missingness mechanism will also have no asymptotic effect on parameter estimation and hypothesis testing. Third, it suggests that Maximum Likelihood inference is better than Multiple Imputation when working with partially-observed datasets with many observations.

The rest of this paper is structured as follows. Section 2 states and proves the two Theorems. After Theorem \ref{thm:thm1} is proved, the first practical implication mentioned above is discussed in more detail. The remaining two are discussed after the proof of Theorem \ref{thm:thm2}. These Theorems are then verified via simulation. To do so, this paper simulates data using a dataset from a previous study on the impact of International Human Rights Law on public support for human rights abuses. Through this simulation, this paper verifies that (1) the average-loglikelihood function (ignoring the missingness mechanism) generated by partial data will approximate that which would have been generated by complete data as sample size increases and missingness decreases, (2) this approximation depends only on sample size for Missing at Random data, (3) this leads to valid inferences from missing data, which will tend to be better than those obtained through Multiple Imputation. 

This paper offers a new theory about the effect of sample size and missingness on inference with missing data. This theory sheds light on the factors that affect the validity of Direct-Likelihood, Bayesian, and Frequentist inferences obtained from missing data that ignore the missingness mechanism, and proves a more general result than has been demonstrated in previous literature. It yields three practical implications that help researchers know when inferences drawn from partial datasets will approximate those obtained from complete ones, the impact (or lack thereof, or "ignorability") of the proportion of missing data in large samples, and the benefits of using Maximum Likelihood inference instead of Multiple Imputation.  

\section{Inference with Missing Data: Formal Explanation}

This section proves that as the extent of missingness decreases and the sample size increases, the average-loglikelihood function (ignoring the missingness mechanism) generated by a partial dataset will converge in probability to the average-loglikelihood function generated by a complete one. Moreover, it shows that if the data are Missing at Random (MAR), then this convergence depends only on the sample size. This result adds to the literature by proving that \textit{all} possible inferences, such as posterior modes, uncertainty estimates, confidence intervals, likelihood ratios, test statistics, and any possible features or quantities that one can draw from a loglikelihood function generated by partial data will approximate their true values (what they would have been without missing data), whereas previous results have shown only that the posterior mode of the partial loglikelihood will tend to approximate its analogue in a complete dataset \citep{little_statistical_2019, little_conditions_2017, seaman_what_2013, takai_asymptotic_2013, nielsen_inference_1997}. The results also imply that if the data are Missing at Random, then the missing data mechanism will have asymptotically no effect on parameter estimation and hypothesis testing. After the proof of each theorem, practical implications are discussed. 

\subsection{Formal Setup}

The setup, including relevant concepts, terminology, and notation, for the formal results in this section is as follows. Let

\begin{enumerate}
\item $X : \Omega \mapsto K_X $ be a multivariate random variable on some probability space, $(\Omega, \mathcal{F}, P)$, onto some compact subset of $\mathbb{R}^m$, $K_X$.
    
    \item $X \sim \mathcal{D}(\theta^*)$ where $\mathcal{D}$ is a continuous probability density function with parameter $\theta^* \in \mathbb{R}^r$.
    
    \item $\{P_k\}_{k =1}^{\infty}$ is a sequence of probability measures on $\{0,1\}^m$ that depend on values of $(X, \psi) \in \mathbb{R}^m \times K_\psi$, where $K_{\psi}$ is some space, and $\{\psi_k\}_{k=1}^{\infty}$ is a corresponding sequence of parameters in $K_{\psi}$ such that $\forall X \in \mathbb{R}^m, M \in \{0,1\}^m$,
    
    $$\lim_{k \to \infty} P_k(M | X,\psi_k) =  \mathbb{1}_{\{\{1\}^m\}}(M) $$
    
    Note that this is a translation of the phrase "as the extent of missingness decreases" into mathematical language. This translation considers a sequence of probability measures that characterize the distribution of the different possible missingness vectors. "As the extent of missingness decreases" centers around this sequence converging \textit{pointwise} to the indicator function that takes on $1$ if the missingness vector is entirely $1$'s (indicating that the data are fully observed), and $0$ otherwise. One can think of this as a kind of weak convergence of measures result. Note also that so far, no assumptions are made about whether the data are Missing at Random. Theorem \ref{thm:thm2} discusses what happens when the Missing at Random assumption holds. 
    
    \item $O(X,M)$ denotes the "subvector" of $X \in \mathbb{R}^m$ that represents values that are observed when given a missingness vector $M \in \{0,1\}^m$ (this notation is also used in \cite{seaman_what_2013}). In particular, $O(X,M)$ is a function of $X$ and $M$ that returns a vector whose $j$th entry is the $j$th entry of $X$ if the $j$th entry of $M$ is $1$, and a "placeholder", $*$, otherwise. For example, $O([1,2,3], [1,0,1]) = [1, *,3]$, where $*$ is a placeholder for the missing value. No vector space operations are performed on $O(X,M)$ directly. Rather, the only significance of $O(X,M)$ is that it is used as an argument of the partial likelihood function, discussed further below in item (7). 
    
    \item $\{F_i\}_{i=1}^\infty$ is a sequence of randomly sampled independent outcomes in $(\Omega, \mathcal{F}, P)$ of non-zero probability.
    
    \item$\{M_i\}_{i=1}^{\infty}$ is a corresponding sequence in $\{0,1\}^m$.
    
    \item $f$ is the corresponding loglikelihood function that \textit{ignores} the missing data mechanism and only depends on $\theta \in K_{\theta}$, a compact subset of $\mathbb{R}^r$. For a complete realization of $X$, $f(X|\theta)$ is equal to the probability of observing $X$ when given $X \sim \mathcal{D}(\theta)$. To evaluate $f$ for vectors with missing observations, the following conditional expectation is taken: $f(O(X,M)| \theta) = \mathbb{E}[f(X^* | \theta) | O(X,M) = O(X^*, M)]$ \citep{little_statistical_2019, seaman_what_2013}.
    \end{enumerate}

Note that in this setup, the parameter is arbitrary. Therefore, Theorem \ref{thm:thm1} will hold if one considers only a subvector of a parameter of interest, and Theorem \ref{thm:thm2} will hold if the data are "Partially Missing at Random" \citep{little_conditions_2017}. 

In this paper, the subscript on the expectation operator indicates the variable with respect to which the expectation is taken. The Theorems assume random variables with continuous distributions, but they can be easily generalized to discrete ones by way of simple functions. 
    
\subsection{Formal Results}

With this setup, 2 important lemmas hold. After stating these lemmas, this paper will provide Theorem \ref{thm:thm1}. All proofs are contained in the Appendix. 

\begin{lemma}
\label{lemma:l1}
$\lim_{k \to \infty} \mathbb{E}_{M \in \{0,1\}^m} \bigg[ f(O(X, M)| \theta) \big| X, \psi_k \bigg] = f(X| \theta)$ almost uniformly on $K_X \times K_\theta$. 
\end{lemma}

\begin{lemma}
\label{lemma:l2}
There exists a sequence of simple functions $\{\phi_n\}_{n=1}^{\infty}$ of the form $\sum_{j=1}^{n}z_j\mathbb{1}_{E_j}$ such that $\lim_{n \to \infty} \phi_n = X$ uniformly. Moreover,

\begin{enumerate}
    \item $f(O(X , M)| \theta)  = f(O(\phi_n , M)| \theta)  + \varepsilon(n)$ uniformly over $K_X \times K_\theta$ for all $ M \in \{0,1\}^m$, \text{and} 
    
    \item $f(X | \theta) = f(\phi_n   | \theta) + \varepsilon(n)$ uniformly over $K_X \times K_\theta$.
\end{enumerate}  

\end{lemma}
Theorem \ref{thm:thm1} will now be stated. The main intuition behind Theorem \ref{thm:thm1} is this. The average-loglikelihood function (ignoring the missingness mechanism) generated by partial data can be viewed as a summation of some function of different realizations of a random variable that are partially observed due to the introduction of a missingness vector. As the sample size increases, more and more of each realization is added, and one can apply the Law of Large Numbers to all the terms of the summation that correspond to that particular realization of the random variable. Essentially, one is splitting up the terms of the summation into different "classes", and applying the Law of Large Numbers to each class. Through this application, one shows that this summation converges to some vector. One can apply the Portmanteau Theorem (i.e., Lemma \ref{lemma:l1}) to show that the value of the resulting vector depends on the missingness mechanism, and that this vector converges to the function of the realization of the random variable that is completely observed as the extent of missingness decreases. 

\begin{theorem}
\label{thm:thm1}
\begin{align*}
    \lim_{N \to \infty} \frac{1}{N}\sum_{i=1}^N f(O(X(F_i), M_i) | \theta) 
    \\
    = \lim_{N \to \infty} \frac{1}{N}\sum_{i=1}^N f(X(F_i) | \theta) + \frac{1}{N}\sum_{i=1}^N \varepsilon_i(k) + \frac{2}{N}\sum_{i=1}^N\varepsilon_i(n)
\end{align*}

in probability on $K_X \times K_\theta$, where $\{\psi_k\}_{k \in \mathbb{N}}$ is a sequence of parameters generating the $\{M_i\}_{i \in \mathbb{N}}$, and $\{\phi_n\}_{n=1}^{\infty}$ is a sequence of simple functions such that $\lim_{n \to \infty} \phi _n = X$ uniformly over $K_\theta$.
\end{theorem}

Prior to discussing the implications of Theorem \ref{thm:thm1}, it is worth explaining the role of the simple function, and $n$ above. The main idea here is that the simple functions capture the idea of \textit{approximation}, so the term $\varepsilon(n)$ can be treated as $0$ for practical purposes. One can approximate a random variable $X$ to some degree using simple functions (since no machine can ever really achieve an approximation of infinite precision, this process accurately reflects how most approximation algorithms work). The loglikelihood functions that are generated from a partial dataset will in some sense "approximate an approximation" (by way of simple functions) of the loglikelihood functions generated by a complete dataset. When one fixes the precision with which they approximate $X$, thereby fixing $n$ finite but sufficiently large, one need not worry about the $n$ and $\varepsilon(n)$ terms in the result. Intuitively, if one wants to have more and more precise approximations, the sample size $N$ just has to be greater than or equal to $n$ in order for the result to hold, which the Theorem suggests. 

Theorem \ref{thm:thm1} shows that the average-loglikelihood function (ignoring the missingness mechanism) generated by partial data will approximate that which would have been generated by complete data as sample size increases and the extent of missingness decreases. This implies that inferences drawn from the likelihood function generated by partial data such as posterior modes, uncertainty estimates, confidence intervals, likelihood ratios, and test statistics will be close to their analogous values from the complete dataset. The posterior modes are obtained by taking the mode of the parameters given the distribution generated by the average-loglikelihood function, and the uncertainty estimates are obtained by measuring the distance between the parameters in this space and the mode given these distributions. Since these distributions tend to be identical in large samples with small amounts of missingness, so will the modes and corresponding uncertainties, and hence the confidence intervals. Indeed, because these distributions will become approximately identical, \textit{any} feature or quantity that can be obtained from the partial-data loglikelihood function will approximate its true value. One practical implication of this is that inferences from partial datasets will tend to be valid in sufficiently large samples with sufficiently small amounts of missing data. The next section proves the interesting result that if the data are Missing at Random, then this approximation depends only on the sample size and not the extent of missingness. Prior to stating and proving Theorem \ref{thm:thm2}, Missing at Random will be defined. It is similar to the definition of "Everywhere Missing at Random" given by \cite{seaman_what_2013}. 

\begin{definition}[Missing at Random]
The data are Missing at Random if 
$$P_k(M | X_1) = P_k(M|X_2) \ \forall M, X_1, X_2 : O(X_1, M) = O(X_2,M) \ \text{almost surely}$$
\end{definition}

This is technically not the same definition that \citep{little_statistical_2019} use to prove the consistency of the posterior mode \citep{seaman_what_2013}, but this does not make a difference when one considers asymptotic inference. The definition they use is "Realized Missing at Random", which is analogous to the "Everywhere Missing at Random" definition, except it holds only for the realized $M$'s. However, it does not make much of a difference for large sample sizes, because as more and more observations are drawn, the space of realized $M$'s will contain all possible $M$'s of nonzero probability. As a result, the application of the "Realized" definition will result in its application to almost every $M$ that can be observed, which makes it akin to the "(Almost) Everywhere" definition. Thus, while this definition is not technically the same as that used to prove previous results about consistency of the posterior mode, it does not pose an issue here because this paper is concerned with very large samples. With this definition in hand, this paper proceeds to Theorem \ref{thm:thm2}. 

The key intuition behind Theorem \ref{thm:thm2} is that for any missingness pattern, the complete data can be partitioned into "classes" that each contain partial observations that will be identical under the specified pattern. For example, if one considers the class of missingness vectors where only the first entry is observed, then $(1,0)$ and $(1,1)$ will belong to the same class generated by that missingness vector. The expectation operator can be taken over these different classes to show that the error induced by the missingness mechanism tends to zero as the sample size increases if the data are Missing at Random.

\begin{theorem}
\label{thm:thm2}
If the data are Missing at Random, then 

\begin{align*}
    \lim_{N \to \infty} \frac{1}{N}\sum_{i=1}^N f(O(X(F_i), M_i) | \theta) = \lim_{N \to \infty} \frac{1}{N}\sum_{i=1}^N f(X(F_i) | \theta) + \frac{2}{N}\sum_{i=1}^N\varepsilon_i(n)
\end{align*}

in probability on $K_X \times K_\theta$ for any parameter $\psi \in K_\psi$ that describes the missingness mechanism from which the $\{M_i\}_{i \in \mathbb{N}}$ are generated, and where $\{\phi_n\}_{n=1}^{\infty}$ is a sequence of simple functions such that $\lim_{n \to \infty} \phi _n = X$ uniformly over $K_\theta$.

\end{theorem}

Theorem \ref{thm:thm2} shows that if the data are Missing at Random, then the convergence of the average-loglikelihood function (ignoring the missingness mechanism) generated by partial data to that which would have been generated by complete data depends only on the sample size and not the extent of missingness. In other words, the effect of the missing data mechanism becomes arbitrarily small if the data are missing at random. This has three key implications for when data are Missing at Random. First, the validity of inferences drawn from partial data depends only on sample size and not the extent of missingness. This means that one could have an incredibly large amount of missing data, and their inferences drawn from the average-loglikelihood would still tend to be valid and consistent for sufficiently large datasets (obviously the more missing data one has, the larger the dataset that one might need). This is because the posterior modes, uncertainty estimates, confidence intervals, likelihood ratios, test statistics, and any other quantities that can be obtained from the likelihood function will get arbitrarily close to their true values no matter the missingness mechanism, so long as the data are Missing at Random. Practically speaking, this means that even if one has 90\% of their data missing, they can still do valid inference on that partial dataset. Any quantity they estimate from the partial loglikelihood will approximate its complete-data counterpart. So far, only the posterior mode has been shown to be consistently estimated in previous literature \citep{little_statistical_2019, little_conditions_2017, seaman_what_2013, takai_asymptotic_2013, nielsen_inference_1997}. This is what makes this result noteworthy.

Second, the missing data mechanism will asymptotically have no effect on parameter estimation and hypothesis testing if the data are Missing at Random. Parameter estimation and hypothesis testing are done using precisely this likelihood function. The former is done analyzing the extremums of the likelihood function, $\frac{1}{N}\sum_{i=1}^N f(O(X(F_i), M_i) | \theta)$ on a set of parameters. Since the missingness mechanism has asymptotically no effect on the function if the data are Missing at Random, it follows that the missingness mechanism has asymptotically no effect on parameter estimation. The latter is usually done by analyzing the set of likelihood ratios per the Wilks Test: 

$$-2\ln \bigg(\frac{\sup_{\theta \in \Theta_0} \frac{1}{N}\sum_{i=1}^N f(O(X(F_i), M_i) | \theta)}{\sup_{\theta \in \Theta} \frac{1}{N}\sum_{i=1}^N f(O(X(F_i), M_i) | \theta)} \bigg)$$ 

where $\theta, \theta_0 \subseteq K_\theta$. Thus, hypothesis testing is also done directly on the likelihood function. Since the missing data mechanism has asymptotically no effect on this function if the data are Missing at Random, it follows also that it has asymptotically no effect on hypothesis testing either.

Finally, when doing inference on large samples with missing data, Maximum Likelihood is preferable to Multiple Imputation, which some studies have recently hinted at \citep{shin_maximum_2017, von_hippel_new_2016}. This is due to a few reasons. First, the average-loglikelihood function (ignoring the missingness mechanism) generated by large partial datasets is going to mirror that which would have been generated by the complete dataset. Multiple Imputation works by first obtaining the average-loglikelihood function generated by partial data, and then imputing multiple values for the missing entries based on this function. Multiple Imputation thus involves an extra, and perhaps harmful step. The average-loglikelihood function is already going to yield parameter estimates, uncertainty estimates, confidence intervals, etc. that are close to those that would have been obtained from complete data, so there is little need to impute missing values to "re-estimate" these quantities. Second, since these "re-estimates" are obtained from datasets with imputed values, one is at increased risk of obtaining biased estimates for the parameters and their corresponding uncertainties. By imputing values, one is introducing additional noise in their analysis of the parameters, which reduces the quality of the estimates. Indeed, the risk of Multiple Imputation leading to bias with respect to certain parameter estimates has been observed in other settings \citep{madley-dowd_proportion_2019, mishra_comparative_2014, jochen_multiple_2013, wu_new_2013, black_missing_2011, knol_unpredictable_2010, leite_performance_2010, schafer_analysis_1997}. One justification for the imputation of values is to account for estimation uncertainty \citep{honaker_amelia_2011}, but there is no need to do this, because Theorems \ref{thm:thm1} and \ref{thm:thm2} show that the uncertainty estimate that one obtains from a large enough partial dataset will already be close to the true uncertainty. Multiple imputation is superfluous in large samples and brings inferences farther from their true values. In the next section, this paper concretely shows that because of this resulting bias in Multiple Imputation, Maximum Likelihood inference will tend to outperform Multiple Imputation in large samples with respect to obtaining parameter estimates and corresponding uncertainties that are close to the truth. 

\section{Inference with Missing Data: Simulation Example}

Following previous studies, this paper now shows by way of simulation how sample size and the extent of missingness can affect the validity of inference from partial data \citep{madley-dowd_proportion_2019, mishra_comparative_2014, jochen_multiple_2013, wu_new_2013, black_missing_2011, knol_unpredictable_2010, leite_performance_2010, schafer_analysis_1997}. This simulation is situated in the theoretical context of a study by \cite{lupu_best_2013} evaluating the impact of International Human Rights Law on public support for human rights abuses. It shows (1) that the maximum distance between the average-loglikelihood function (ignoring the missingness mechanism) generated by incomplete data and that which would have been generated by complete data decreases with respect to sample size and missingness, and (2) that for data that is Missing at Random, this approximation depends only on sample size. It further shows this convergence of the average-loglikelihood functions results in parameter estimates and corresponding uncertainties that are close to their true values (what they would have been without missingness), and much closer than those that are obtained with Multiple Imputation. Altogether, these results show that the theorems proven above do hold in practice. 

The data used in this part of the paper are purely simulated. However, to make better sense of the effects of MAR and Missing not at Random (MNAR) data on statistical inference, the simulated data are treated as if they came from the empirical survey study done by \cite{lupu_best_2013}. It is easier to make sense of MAR and MNAR data in this concrete way, because whether data end up being MAR or MNAR depends on the substance of what is being measured. For example, survey data on depression can be MAR if male participants are less likely to answer questions than female participants because of an aversion to showing vulnerability. In this example, the data are MAR because of some substantive reason about the differences in how males and females relate to mental health. It is therefore helpful to situate the simulated data in some empirical context like the \cite{lupu_best_2013} study so that the MAR and MNAR mechanisms are better understood. This is precisely the approach taken in other studies on missing data: \cite{mishra_comparative_2014} simulated data from a diabetic clinical trial, \cite{jochen_multiple_2013} simulated data from an existing dataset on parent-child relationships, \cite{wu_new_2013} simulated data from an existing dataset on dietary habits, and \cite{black_missing_2011} simulated data from an existing dataset on socioeconomic status and student achievement (other studies that do this include \cite{jakobsen_when_2017}, \cite{sung_monte_2007}, \cite{tang_efficient_2012}). Situating data in some empirical example does not alter the simulated nature of the data in any way, it is merely a means of framing it. Thus, the overall simulation results do not depend on the empirical example used. 

\subsection{Theoretical Context of the Simulation}
The \cite{lupu_violence_2019} study analyzes whether international law affects public support for human rights abuses. Through a survey experiment across 3 countries, \cite{lupu_best_2013} estimates the average treatment effect of whether an individual would support their government if told that their government has violated international human rights law. The study found that in India, an individual would tend to support their government less upon hearing this information, whereas in Israel, an individual would tend to support their government more. 

For simplicity, the dataset from this study as it is used in this paper contains only on the India subsample. It includes information about the following for each individual: whether they were told that their country has violated international human rights law, gender, age, ideology, education level, whether they participated in political activism, their level of political interest, and their level of approval for their government. 

\subsection{Methodology}
The main objects of interest in this paper are the average-loglikelihood functions (ignoring the missingness mechanism) generated by partial and complete datasets. To study these objects, this paper uses simulated normally distributed data. The main advantage of focusing on normally distributed data is that it allows for the convenient use of the Norm software \citep{schafer_norm_2013} in R to evaluate likelihoods of partial observations. This is useful because obtaining likelihoods from partial data can be computationally complex \citep{little_statistical_2019, sung_monte_2007, honaker_amelia_2011}. Although this paper focuses on likelihood functions that are generated by normally distributed data, the Theorems above hold for any kind of data with a continuous probability distribution so long as the likelihood function corresponds to that distribution. 

To simulate normally distributed data from the International Human Rights Law dataset, both the mean vector and variance-covariance matrix of the original dataset are calculated. These parameters are then used to simulate $500$ observations according to a multivariate normal model, $\mathcal{N}(\theta)= \mathcal{N}(\mu, \Sigma)$, that are treated as a perfectly representative sample of the population, from which simulated samples of the population can be drawn. Simulated samples from size $N=30$ to $N = 110$ are taken in increments of 1. For each sample, a distribution of $(\mu, \Sigma)$, $K$, is obtained by bootstrapping: taking $100$ "subsamples" of size $N$ of each $N$-sized sample, and calculating $(\mu,\Sigma)$ for each subsample. For each simulated sample, Missing at Random (MAR) data and Missing not at Random (MNAR) data are introduced in varying degrees. To introduce the former, it removes observations under the assumption that people with less interest in politics are less likely to answer the question about whether they approve of their government. To introduce the latter, it removes observations under the assumption that people who generally approve of the government tend not to indicate this on their surveys. In each case, missingness is introduced according to a Bernoulli distribution. The paper looks at what happens when $20\%, 30\%$, and $40\%$ of a specific variable (the subject's approval of their government) are MAR/MNAR. Then, the average-loglikelihood functions generated by the complete and incomplete datasets are approximated, using the Norm package to evaluate the average-loglikelihoods of partial data. These functions are evaluated over the space of parameters generated by bootstrapping, $K$, and the supremum (which is a maximum given that all of this is done on a compact set) of the differences between the two functions over $K$ is evaluated. 

The sample sizes considered for this part of the empirical analysis, namely the analysis of the average-loglikelihood functions, are relatively small; however, this is done intentionally. The average-loglikelihood functions considered here can yield values that are too small for R to return as non-zero numbers. One could work with the loglikelihoods instead, but this can also be difficult because one would have to take the log of the difference of two average-loglikelihood functions, and the average-loglikelihood functions will still replace some small values with zero, leading to a difference of zero and thus producing $-\infty$. The reason one would have to take the log of the difference of the functions rather than the difference of log of the two functions is that as sample size increases, the average-loglikelihoods tend to get really small; and the smaller they are, the greater the effect of a slight shift in the input (this is because the derivative of the log is $\frac{1}{x})$. Thus, differences between the two functions will tend to be exaggerated as the sample size increases, making the effects of sample size and the missingness mechanism harder to detect. To ensure that effects are correctly and precisely measured, the sample sizes are kept relatively small. This does not prove to be a big issue, since the effects of sample size and the extent of missingness are still clear with these small sample sizes. 

Afterward, this paper illustrates the effect that this convergence of the average-loglikelihood functions has on the Maximum Likelihood Estimates (MLEs) and their corresponding uncertainties. Using the likelihood functions generated by partial data, this paper obtains the MLEs and corresponding uncertainties of the parameter, $(\mu, \Sigma)$, and compares them to the estimates and corresponding uncertainties of these parameters that are obtained through Multiple Imputation by way of the Amelia software in R \citep{honaker_amelia_2011}. This is done by leveraging Amelia's Bootstrap to obtain a distribution of estimates for the parameters $(\mu,\Sigma)$ whose size is equal to the number of imputations, which is set to $m=200$. This comparison is done over samples ranging from size $N=200$ to $N = 5000$ under different Missing at Random mechanisms. The Missing at Random mechanisms are simulated in the same way as above for $\psi = 0.1, 0.3,$ and $0.7$.  

\subsection{Results}
Figure \ref{fig:supmar} illustrates how the supremum of the differences between the average-loglikelihood functions (ignoring the missingness mechanism) generated by partial and complete data changes with respect to sample size and the extent of Missing at Random data. Figure \ref{fig:supmnar} is analogous for Missing not at Random data. Figure \ref{fig:lkversusmi.mle} shows how the parameter estimates generated by the likelihood function approximate the true parameter as sample size increases under the Missing at Random assumption, and compares this approximation to that obtained by Multiple Imputation. Figure \ref{fig:lkversusmi.uncertainty} shows this same phenomenon with respect to the uncertainty of the parameter estimates.  

\begin{figure}[hbt!]
\includegraphics[scale = 0.25]{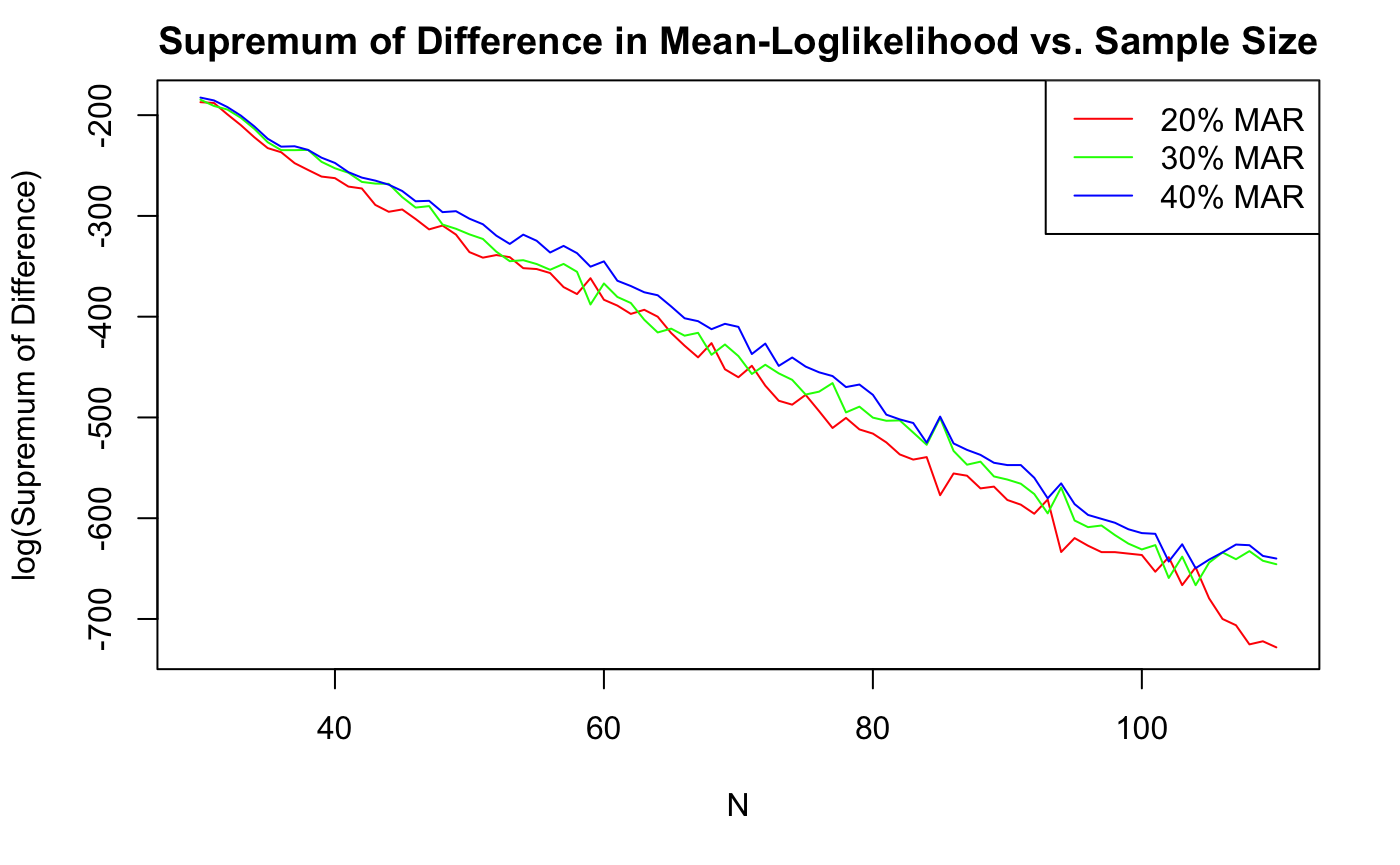}
\centering
\caption{Supremum of Difference in Average-Loglikelihood (log) vs. Sample Size when data are MAR}
\label{fig:supmar}
\end{figure}

\begin{figure}[hbt!]
\includegraphics[scale = 0.25]{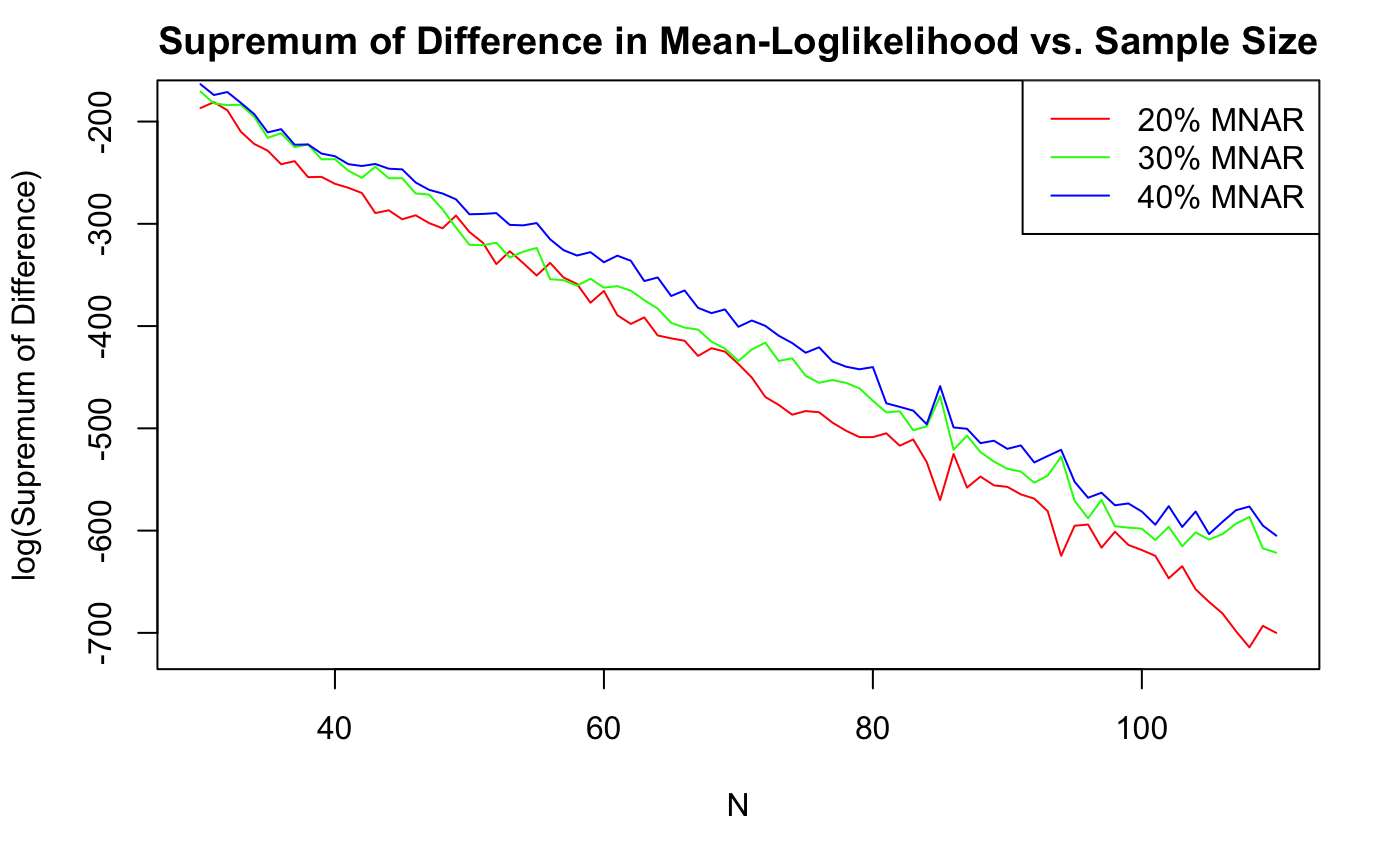}
\centering
\caption{Supremum of Difference in Average-Loglikelihood (log) vs. Sample Size when data are MNAR}
\label{fig:supmnar}
\end{figure}

\begin{figure}[hbt!]
\includegraphics[scale = 0.5]{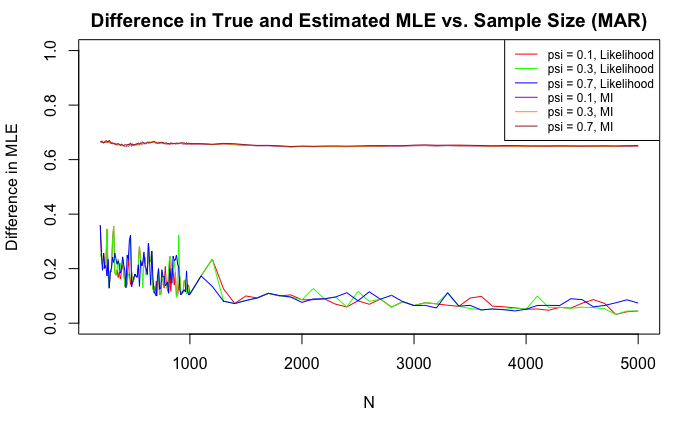}
\centering
\caption{Difference in MLE vs. Sample Size under Likelihood and MI}
\label{fig:lkversusmi.mle}
\end{figure}

\begin{figure}[hbt!]
\includegraphics[scale = 0.5]{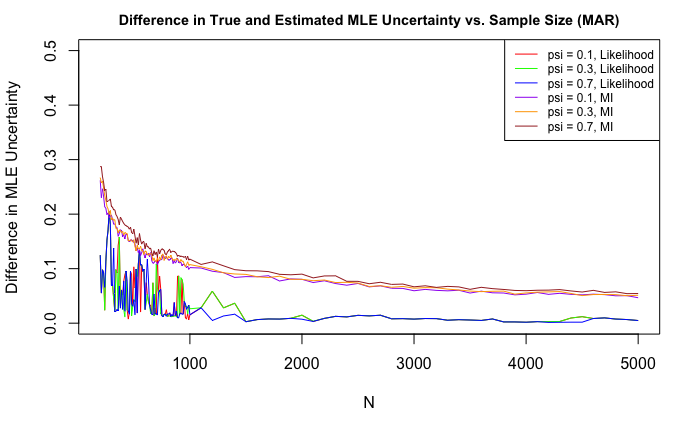}
\centering
\caption{Difference in MLE Uncertainty vs. Sample Size under Likelihood and MI}
\label{fig:lkversusmi.uncertainty}
\end{figure}

\subsection{Discussion}
Figures \ref{fig:supmar} and \ref{fig:supmnar} show that as the sample size increases, the maximum distance between the average-loglikelihood functions (ignoring the missingness mechanism) generated by partial and complete data tends to decrease. They also show that as this maximum distance tends to decreases as the extent of missingness decreases; however, this effect appears to be more pronounced in Missing not at Random data than in Missing at Random data. This is expected, since Theorem \ref{thm:thm2} proves that for Missing at Random data, only the sample size matters for convergence. It may also be interesting to note that the value of the supremum of the differences between average-loglikelihood functions generated by partial and complete data tends to be the same given similar sample sizes and degrees of missingness. Overall, these results verify that as the sample size increases and the extent of missingness decreases, a average-loglikelihood function generated by incomplete data will approximate that which would have been generated by complete data. 

Figures \ref{fig:lkversusmi.mle} and \ref{fig:lkversusmi.uncertainty} show further that the inferences drawn from partial datasets under the Missing at Random assumption will tend to approximate the truth as sample size increases. The difference between the MLE generated by Maximum Likelihood on partial datasets and the true MLE decreases as sample size increases. The same is true for the MLE uncertainty. Moreover, in both cases these differences are smaller than the differences that would result from using Multiple Imputation instead of Maximum Likelihood under different sample size and missingness specifications. This suggests that not only is Likelihood good at approximating the true parameters and corresponding uncertainties, it also does a better job than Multiple Imputation

\section{Conclusion}
Missing data are often an inevitable and difficult part of scientific research. Researchers have tried to deal with this problem by doing inference on partial datasets. However, it is not always clear when inferences done on partial datasets will be asymptotically valid. This paper fills in this knowledge gap by offering a new asymptotic theory about inference with missing data. It proves that as sample size increases and the extent of missingness decreases, the average-loglikelihood function (ignoring the missingness mechanism) generated by a partial dataset will approximate that generated by a complete dataset. Moreover, it shows that if the data are Missing at Random, this approximation depends only on sample size. This result adds to previous research by proving that \textit{all} possible inferences that one can draw from a loglikelihood function generated by partial data will approximate their true values (their complete-data analogues), whereas previous research has only proved this to be true for the posterior modes \citep{little_statistical_2019, little_conditions_2017, seaman_what_2013, takai_asymptotic_2013, nielsen_inference_1997}. This result has important practical implications: (1) the inferences obtained from partial datasets will approximate those that would have been obtained from complete ones for sufficiently large samples and small amounts of missing data, (2) when the data are Missing at Random, only the sample size matters for the validity of these inferences, not the proportion of missing data, (3) if the data are Missing at Random, then the missingness mechanism will have no asymptotic effect on parameter estimation and hypothesis testing, and (4) Maximum Likelihood inference is preferable to Multiple Imputation in large samples. 
\section{Appendix}

\textit{Proof of Lemma \ref{lemma:l1}}

\textbf{Lemma \ref{lemma:l1}:} \textit{$\lim_{k \to \infty} \mathbb{E}_{M \in \{0,1\}^m} \bigg[ f(O(X, M)| \theta) \big| X, \psi_k \bigg] = f(X| \theta)$ almost uniformly on $K_X \times K_\theta$.}

\begin{proof}
Let $\theta \in K_\theta$, $X \in K_X$. 

\begin{align*}
      \lim_{k \to \infty}  \mathbb{E}_{M \in \{0,1\}^m} \bigg[ f(O(X, M)| \theta) \big| X, \psi_k \bigg] \nonumber
    \\
    = \lim_{k \to \infty} \int_{\{0,1\}^m} f(O(X,M) | \theta) dP_k  \nonumber
\end{align*}

Observe that $f(O(X,M) | \theta)$ is bounded and Lipschitz over $\{0,1\}^m$. Observe also that $\forall M \in \{0,1\}^m, \limsup_{k \to \infty}P_k(M | X,\psi_k)  \leq \mathbb{1}_{\{\{1\}^m\}}(M)$ by construction of $\{P_k\}_{k=1}^\infty$. Therefore, by the Portmanteau Theorem,

\begin{align}
    \lim_{k \to \infty} \int_{\{0,1\}^m} f(O(X,M) | \theta) dP_k  \nonumber
    \\
    = \int_{\{0,1\}^m} f(O(X,M) | \theta) d\mathbb{1}_{\{\{1\}^m\}}
    \\
    = f(X| \theta) \nonumber
    \\
    \implies \lim_{k \to \infty} \mathbb{E}_{M \in \{0,1\}^m} \bigg[ f(O(X, M)| \theta) \big| X, \psi_k \bigg] = f(X| \theta) \ \text{for a given $X, \theta$} \nonumber
\end{align}

Since $K_X \times K_\theta$ is compact, it is of finite measure. Therefore, by Egorov's Theorem, this convergence happens almost uniformly on $K_X \times K_\theta$: 

\begin{align*}
    \mathbb{E}_{M \in \{0,1\}^m} \bigg[ f(O(X, M)| \theta) \big| X, \psi_k \bigg] \ \text{converges almost uniformly on $K_X \times K_\theta$ to } \ f(X| \theta)
\end{align*}
\end{proof}

\textit{Proof of Lemma 2.}

\textbf{Lemma \ref{lemma:l2}:} \textit{There exists a sequence of simple functions $\{\phi_n\}_{n=1}^{\infty}$ of the form $\sum_{j=1}^{n}z_j\mathbb{1}_{E_j}$ such that $\lim_{n \to \infty} \phi_n = X$ uniformly. Moreover,}

\begin{enumerate}
    \item $f(O(X , M)| \theta)  = f(O(\phi_n , M)| \theta)  + \varepsilon(n)$ uniformly over $K_X \times K_\theta$ for all $ M \in \{0,1\}^m$, \text{and} 
    
    \item $f(X | \theta) = f(\phi_n   | \theta) + \varepsilon(n)$ uniformly over $K_X \times K_\theta$.
\end{enumerate} 

\begin{proof}
Let $M \in \{0,1\}^m$. $f(O(X , M)| \theta) $ is continuous on $K_X \times K_\theta$, which is a compact set. So, $f(O(X , M)| \theta) $ is uniformly continuous on $K_X \times K_\theta$. This means that $\forall \varepsilon > 0, \exists \delta(M, \varepsilon) > 0 : \forall X_1, X_2 \in K_X : || X_1 - X_2|| < \delta(M, \varepsilon),$ 

\begin{align}
    |f(O(X_1, M)| \theta)  - f(O(X_2, M)| \theta) | < \varepsilon  \nonumber
\end{align}

By Theorem 2.10 in \cite{folland_real_2007}, $\exists \{\phi_n\}_{n = 1}^{\infty}$, a sequence of simple functions of the form $\sum_{j =1 } ^{n} z_j \mathbb{1}_{E_j}$ such that $\lim_{n \to \infty} \phi _n = X$ uniformly. This implies that $\forall \tilde{\varepsilon} >0, \exists \tilde{N} \in \mathbb{N} : \forall n \geq \tilde{N}, \sup_{F \in X^{-1}(K_X)} |X(F) - \phi_n(F)| < \tilde{\varepsilon}$. Here $X^{-1}(K_X)$ can be thought of as the space of outcomes that, when input as arguments into $X$, return values that span $K_X$.

So, if $\tilde{\varepsilon} = \delta(M, \varepsilon), \exists N_1 \in \mathbb{N} : \forall n \geq N_1,$
\begin{align}
\label{eqn:conv1}
\sup_{F \in X^{-1}(K_X)}|X(F) - \phi_n(F)| < \delta(M,\varepsilon)
\\
\label{eqn:conv2}
\implies 
    |f(O(X,M) | \theta) - f(O(\phi_n,M) | \theta) | < \varepsilon
    \\
   \implies  f(O(X,M) | \theta) = f(O(\phi_n,M) | \theta) + \varepsilon(n)  \nonumber
\end{align}

uniformly over $K_X \times K_\theta$ for the given $M \in \{0,1\}^m$. Since $M$ can only adopt finitely many values in $\{0,1\}^m$, $\forall \varepsilon >0$ $\exists \delta = \min_{M \in \{0,1\}^m} \{\delta(M,\varepsilon) \}$, and hence a corresponding $N_1^*$ analogous to $N_1$ in Equations \ref{eqn:conv1} and \ref{eqn:conv2} such that: 

\begin{align}
\label{eqn:simp1}
    f(O(X,M) | \theta) = f(O(\phi_n,M)  | \theta) + \varepsilon(n) 
\end{align}

uniformly over $K_X \times K_\theta$ for all $M\in \{0,1\}^m$. Letting $M = \{1\}^m$ shows that: 

\begin{align}
\label{eqn:simp2}
    f(X| \theta) = f(\phi_n  | \theta) + \varepsilon(n)
\end{align}

uniformly over $K_X \times K_\theta$.
\end{proof}

\textit{Proof of Theorem 1.}

\textbf{Theorem \ref{thm:thm1}:} \begin{align*}
    \lim_{N \to \infty} \frac{1}{N}\sum_{i=1}^N f(O(X(F_i), M_i) | \theta) 
    \\
    = \lim_{N \to \infty} \frac{1}{N}\sum_{i=1}^N f(X(F_i) | \theta) + \frac{1}{N}\sum_{i=1}^N \varepsilon_i(k) + \frac{2}{N}\sum_{i=1}^N\varepsilon_i(n)
\end{align*}

\textit{in probability on $K_X \times K_\theta$, where $\{\psi_k\}_{k \in \mathbb{N}}$ is a sequence of parameters generating the $\{M_i\}_{i \in \mathbb{N}}$, and $\{\phi_n\}_{n=1}^{\infty}$ is a sequence of simple functions such that $\lim_{n \to \infty} \phi _n = X$ uniformly over $K_\theta$.}

\begin{proof}

Begin by noting that, uniformly over $K_\theta$,

\begin{align}
    \sum_{i=1}^N f(O(X(F_i), M_i) | \theta)  \nonumber
    \\
    = \sum_{i=1}^N \bigg[f(O(\phi_n(F_i), M_i) | \theta) + \varepsilon_i(n) \bigg] \ \text{by Lemma \ref{lemma:l2}} \nonumber
\end{align}

Now, note that given a sequence of simple functions, $\{\phi_n\}_{n=1}^{\infty}$ of the form $\sum_{j =1}^n z_j \mathbb{1}_{E_j}$, such that $\lim_{n \to \infty} \phi _n = X$ uniformly,

\begin{align}
\label{eqn:simpfun}
    f(O(\phi_n, M) | \theta) 
    = \sum_{j=1}^n f(O(z_j, M)  | \theta) \mathbb{1}_{E_j}
\end{align}
for all $M \in \{0,1\}^m, \theta \in K_\theta$. 

Thus, uniformly over $K_\theta$,

\begin{align}
    \sum_{i=1}^N \bigg[f(O(\phi_n(F_i), M_i) | \theta) + \varepsilon_i(n) \bigg] \nonumber
    \\
    =\sum_{i=1}^N \sum_{j =1}^n \bigg[ f(O(z_j, M_i) | \theta)\mathbb{1}_{E_j}(F_i) + \varepsilon_{i}(n)\mathbb{1}_{E_j}(F_i) \bigg] \ \text{by Equation \ref{eqn:simpfun}} \nonumber
\end{align}

Now, note that the $\{E_j\}_{j=1}^n$ are mutually exclusive and collectively exhaustive of $X^{-1}(K_X)$ by properties of simple functions, so $\forall i \in \{1,...,N\}$, $\exists ! j(i) \in \{1,...,n\} : F_i \in E_{j(i)}$. This implies that $\forall j \in \{1,..., n\}, \exists ! I(j) \subseteq \{1,...,N\} : \forall i \in I(j), F_i \in E_{j}$, and $I(j_1) \cap I(j_2) \neq \emptyset$ if and only if $j_1 = j_2$. This is a formal way of saying that every outcome in $\{F_i\}_{i =1}^{N}$ is associated with some $z_j$ in the construction of $\phi_n$, and vice versa. Thus, uniformly over $K_\theta$,

\begin{align}
    \sum_{i=1}^N \sum_{j =1}^n \bigg[ f(O(z_j, M_i) | \theta)\mathbb{1}_{E_j}(F_i) + \varepsilon_{i}(n)\mathbb{1}_{E_j}(F_i) \bigg] \nonumber
    \\
    = \sum_{j=1}^n \sum_{l \in I(j)} \bigg[ f(O(z_j, M_l) | \theta) + \varepsilon_{l}(n) \bigg]  \nonumber
\end{align}

By the Strong Law of Large Numbers and Lemma \ref{lemma:l1}, for any $j \in \{1,...,n\}$,

\begin{align}
    \frac{1}{|I(j)|} \sum_{l \in I(j)} f(O(z_j, M_l)| \theta)
    = \mathbb{E}_{M \in \{0,1\}^m} \bigg[ f(O(\phi_n, M) | \theta) \bigg| z_j, \psi_k \bigg] + \varepsilon_j(|I(j)|, \theta)  \nonumber
    \\ = f(z_j | \theta) + \varepsilon_j(k) + \varepsilon_j( |I(j)|) \nonumber
\end{align}

almost uniformly over $K_X \times K_\theta$. 

Further, note that $\forall j \in \{1,...,n\}, \varepsilon >0, \exists N_*(j) \in \mathbb{N} : \forall N \geq N_*(j), |\varepsilon_j(|I(j)|)| < \varepsilon $, because $P(F_i) > 0 \ \forall i \in \{1,..., N\}$. So, $\forall \varepsilon > 0, \exists N_* = \max_{j \in \{1,...,n\}} N_*(j) : \forall N \geq N_*, |\varepsilon_j(|I(j)|)| < \varepsilon.$ This is a formal way of saying that for any $j \in \{1,...,n\}$, the corresponding set of outcomes $I(j)$ that, when input as arguments into the random variable $X$, would produce $z_j$, depends on $N$. This is because each outcome has nonzero probability, so increasing $N$ (the sample size), and then drawing random samples will lead to a larger $|I(j)|$ for all $j$. This means that the error term, $\varepsilon_j(|I(j)|)$ becomes arbitrarily small as the sample size, $N$, increases. Formally, 

\begin{align}
\label{eqn:ij}
    \frac{1}{|I(j)|} \sum_{l \in I(j)} f(O(z_j, M_l) | \theta)
    = f(z_j | \theta) + \varepsilon_j(k) + \varepsilon_j(N)
\end{align}

almost uniformly on $K_X \times K_\theta$. 

Thus, the following convergence happens in probability on $K_X \times K_\theta$:

\begin{align}
    = \sum_{j=1}^n \sum_{l \in I(j)} \bigg[ f(O(z_j, M_l) | \theta) + \varepsilon_{l}(n) \bigg]  \nonumber
    \\
    = \sum_{j=1}^n |I(j)| \bigg[ f(z_j | \theta) + \varepsilon_j(k) + \varepsilon_j(N) \bigg] + \sum_{i=1}^N\varepsilon_i(n) \ \text{by Equation \ref{eqn:ij}} \nonumber
    \\
    = \sum_{j=1}^n \bigg[|I(j)| f(z_j|\theta) + |I(j)|\varepsilon_j(k) + |I(j)|\varepsilon_j(N) \bigg]  + \sum_{i=1}^N\varepsilon_i(n) \nonumber
    \\
    = \sum_{i=1}^N \sum_{j=1}^n f(z_j| \theta)\mathbb{1}_{E_j}(F_i) + \sum_{i=1}^N \sum_{j=1}^n \varepsilon_j(k) \mathbb{1}_{E_j}(F_i) + \sum_{i=1}^N \sum_{j=1}^n \varepsilon_j(N) \mathbb{1}_{E_j}(F_i) + \sum_{i=1}^N\varepsilon_i(n) \nonumber
    \\
    = \sum_{i=1}^N f(\phi_n(F_i) | \theta) +  \sum_{i =1}^N\varepsilon_i(k) + \sum_{i =1}^N \varepsilon_i(N) + \sum_{i=1}^N\varepsilon_i(n)  \nonumber
    \\
    \label{eqn:result}
    =\sum_{i=1}^N \bigg[ f(X(F_i) | \theta) + \varepsilon_i(n) \bigg] +  \sum_{i =1}^N\varepsilon_i(k) + \sum_{i =1}^N \varepsilon_i(N) + \sum_{i=1}^N\varepsilon_i(n)  \ \text{by Lemma \ref{lemma:l2}}  \nonumber
    \\
    =\sum_{i=1}^N f(X(F_i) | \theta) + \sum_{i =1}^N\varepsilon_i(k) + \sum_{i =1}^N \varepsilon_i(N) + 2\sum_{i=1}^N\varepsilon_i(n) \nonumber
    \end{align}

This implies that:
\begin{align}
    \lim_{k \to \infty, N \to \infty} \frac{1}{N}\sum_{i=1}^N f(O(X(F_i), M_i) | \theta)  \nonumber
    \\
    = \frac{1}{N} \bigg[\sum_{i=1}^N f(X(F_i) | \theta) +  \sum_{i =1}^N\varepsilon_i(k) + \sum_{i =1}^N \varepsilon_i(N)  + 2\sum_{i=1}^N\varepsilon_i(n)  \bigg] \nonumber
    \\
    = \frac{1}{N}\sum_{i=1}^N f(X(F_i) | \theta) + \frac{1}{N}\sum_{i =1}^N\varepsilon_i(k) + \frac{1}{N}\sum_{i =1}^N \varepsilon_i(N)   + \frac{2}{N}\sum_{i=1}^N\varepsilon_i(n)  \nonumber
\end{align}

in probability on $K_X \times K_\theta$, where $ \lim_{k\to \infty} \sup_{i \in \mathbb{N}}\varepsilon_i(k) = \lim_{N \to \infty} \sup_{i \in \mathbb{N}} \varepsilon_i(N)  =0$. Thus, 

\begin{align*}
    = \lim_{N \to \infty} \frac{1}{N}\sum_{i=1}^N f(X(F_i) | \theta) + \frac{1}{N}\sum_{i=1}^N \varepsilon_i(k) + \frac{2}{N}\sum_{i=1}^N\varepsilon_i(n)
\end{align*}

in probability on $K_X \times K_\theta$, where $\{\psi_k\}_{k \in \mathbb{N}}$ is a sequence of parameters generating the $\{M_i\}_{i \in \mathbb{N}}$, and $\{\phi_n\}_{n=1}^{\infty}$ is a sequence of simple functions such that $\lim_{n \to \infty} \phi _n = X$ uniformly over $K_\theta$.

\end{proof}

\textit{Proof of Theorem 2.}

\textbf{Theorem \ref{thm:thm2}:} \textit{If the data are Missing at Random, then }

\begin{align*}
    \lim_{N \to \infty} \frac{1}{N}\sum_{i=1}^N f(O(X(F_i), M_i) | \theta) = \lim_{N \to \infty} \frac{1}{N}\sum_{i=1}^N f(X(F_i) | \theta) + \frac{2}{N}\sum_{i=1}^N\varepsilon_i(n)
\end{align*}

\textit{in probability on $K_X \times K_\theta$ for any parameter $\psi \in K_\psi$ that describes the missingness mechanism from which the $\{M_i\}_{i \in \mathbb{N}}$ are generated, and where $\{\phi_n\}_{n=1}^{\infty}$ is a sequence of simple functions such that $\lim_{n \to \infty} \phi _n = X$ uniformly over $K_\theta$.}

\begin{proof}
Let $\phi_n$ be a simple function approximation of $X$ as discussed in Theorem \ref{thm:thm1}. It is sufficient to show that:

$$\mathbb{E}_{K_{X}} \bigg[\varepsilon(k)\bigg] = 0$$

since as $N$ approaches infinity, $\frac{1}{N} \sum_{i=1}^{N} \varepsilon_i(k) \approx \mathbb{E}_{K_{X}} \bigg[\varepsilon_i(k)\bigg]$ by the Strong Law of Large Numbers. Note that

\begin{align}
    \mathbb{E}_{K_{X}} \bigg[ \varepsilon_i(k) \bigg]  \nonumber
    \\
    = \mathbb{E}_{K_{X}} \bigg[ \mathbb{E}_{M \in \{0,1\}^m} \bigg[f(O(\phi_n , M)| \theta)  \bigg] - f(\phi_n | \theta) \bigg]  \nonumber
\end{align}

Thus, it is sufficient to show that 

$$\mathbb{E}_{K_{X}} \bigg[ \mathbb{E}_{M \in \{0,1\}^m} \bigg[f(O(\phi_n , M)| \theta)  \bigg] \bigg] = \mathbb{E}_{K_{X}} \bigg[ f(\phi_n | \theta) \bigg]$$

Observe that 

\begin{align}
    \mathbb{E}_{K_{X}} \bigg[ \mathbb{E}_{M \in \{0,1\}^m} \bigg[f(O(\phi_n , M)| \theta)  \bigg] \bigg]  \nonumber
    \\
    = \sum_{j =1}^{n} \sum_{M \in \{0,1\}^m}  f(O(z_j,M) P_k(M|O(z_j, M), \psi_k) P(E_j) \nonumber
    \\
    = \sum_{M \in \{0,1\}^m} \sum_{j =1}^{n} f(O(z_j,M) P_k(M|O(z_j, M), \psi_k) P(E_j) \nonumber
\end{align}

Because the data are Missing at Random, note that $\forall M \in \{0,1\}^m$, there is a set of subsets of $\{1,...,n\}$, denoted $C(M) = \{C(M)_1,...,C(M)_{n(M)}\}$ that are mutually exclusive and collectively exhaustive of $\{1,...,n\}$ such that $O(z_r, M) = O(z_s, M)$ if and only if $\exists! q \in \{1,...,n(M)\} : r,s \in C(M)_q$. This is a formal way of saying that for any given missingness vector, one can partition the space of $\{E_1,...,E_n\}$ into clusters of $E_j$'s that would have the same observed values under the missingness vector given. For any $M \in \{0,1\}^m, q \in \{1,...,n(M)\},$ let the "representative" of $C(M)_q$ be $z(M,q) = O(z_j, M)$, where $j \in C(M)_q$.

Thus, 

\begin{align*}
    \sum_{M \in \{0,1\}^m} \sum_{j =1}^{n} f(O(z_j,M) P_k(M|O(z_j, M), \psi_k) P(E_j) 
    \\
    = \sum_{M \in \{0,1\}^m} \bigg[ \sum_{q = 1}^{n(M)} \bigg[ \sum_{j \in C(M)_q} \bigg[  \ldots \\ \ldots f(O(z_j,M) | \theta) P_k(M | O(z_j,M), \psi_k) P(E_j) \bigg] \bigg] \bigg] 
    \\
    = \sum_{M \in \{0,1\}^m} \bigg[ \sum_{q = 1}^{n(M)} \bigg[ \sum_{j \in C(M)_q} \bigg[  \ldots \\ \ldots \bigg[  \int_{\{E_j : j \in C(M)_q \}} f(z_j|\theta) \frac{dP(E_j)}{P(\{E_j : j \in C(M)_q \})} \bigg] P_k(M | O(z_j,M), \psi_k) P(E_j) \bigg] \bigg] \bigg] 
    \end{align*}
    since $f(O(X,M))$ is obtained by "integrating out" the missing values by taking a conditional expectation. 
    \begin{align*}
    = \sum_{M \in \{0,1\}^m} \bigg[ \sum_{q = 1}^{n(M)} \bigg[ \sum_{j \in C(M)_q} \bigg[  \ldots \\ \ldots \mathbb{E}_{K_{X}} \bigg[ f(\phi_n|\theta) \bigg| O(\phi_n, M) = z(M,q) \bigg] P_k(M | O(z_j,M), \psi_k) P(E_j) \bigg] \bigg] \bigg] 
    \\
     = \sum_{M \in \{0,1\}^m} \bigg[ \sum_{q = 1}^{n(M)} \bigg[ \sum_{j \in C(M)_q} \bigg[ \ldots \\ \ldots \mathbb{E}_{K_{X}} \bigg[ f(\phi_n|\theta) \bigg| j \in C(M)_q,M \bigg] P_k(M | O(z_j,M), \psi_k) P(E_j) \bigg] \bigg] \bigg] 
    \\
     = \sum_{M \in \{0,1\}^m} \bigg[ \sum_{q = 1}^{n(M)} \bigg[ \mathbb{E}_{K_{X}} \bigg[ \ldots \\ \ldots f(\phi_n|\theta) \bigg| j \in C(M)_q,M \bigg] P_k(M | O(z_j,M), \psi_k) \sum_{j \in C(M)_q } P(E_j) \bigg] \bigg] \nonumber
     \\
     \text{by the fact that the data are Missing at Random} \nonumber
     \\
     = \sum_{M \in \{0,1\}^m} \bigg[ \sum_{q = 1}^{n(M)} \bigg[ \mathbb{E}_{K_{X}} \bigg[ \ldots \\ \ldots f(\phi_n|\theta) \bigg| j \in C(M)_q,M \bigg] P_k(M |O(z_j,M), \psi_k) P(j \in C(M)_q | M) \bigg] \bigg] \nonumber
     \\
     = \sum_{M \in \{0,1\}^m} \bigg[ \sum_{q = 1}^{n(M)} \bigg[ \mathbb{E}_{K_{X}} \bigg[ \ldots \\ \ldots f(\phi_n|\theta) \bigg| j \in C(M)_q,M \bigg]P(j \in C(M)_q | M) P_k(M | O(z_j,M) , \psi_k)  \bigg] \bigg] \nonumber
     \\
     = \mathbb{E}_{K_{X}} \bigg[ f(\phi_n|\theta) \bigg] \nonumber
\end{align*}
\end{proof}

\section{Acknowledgements}

This project has not been funded. 

\bibliography{references}

\begin{thebibliography}{}

\bibitem[Black et~al., 2011]{black_missing_2011}
Black, A.~C., Harel, O., and Betsy~McCoach, D. (2011).
\newblock Missing data techniques for multilevel data: implications of model
  misspecification.
\newblock {\em Journal of Applied Statistics}, 38(9):1845--1865.
\newblock Publisher: Taylor \& Francis \_eprint:
  https://doi.org/10.1080/02664763.2010.529882.

\bibitem[Bobb et~al., 2011]{bobb_multiple_2011}
Bobb, J.~F., Scharfstein, D.~O., Daniels, M.~J., Collins, F.~S., and Kelada, S.
  (2011).
\newblock Multiple {Imputation} of {Missing} {Phenotype} {Data} for {QTL}
  {Mapping}.
\newblock {\em Statistical Applications in Genetics and Molecular Biology},
  10(1).
\newblock Publisher: De Gruyter.

\bibitem[Folland, 2007]{folland_real_2007}
Folland, G.~B. (2007).
\newblock {\em Real {Analysis}: {Modern} {Techniques} and {Their}
  {Applications}}.
\newblock Wiley, New York, 2nd edition edition.

\bibitem[George et~al., 2015]{george_effect_2015}
George, S.~V., Ambika, G., and Misra, R. (2015).
\newblock Effect of data gaps on correlation dimension computed from light
  curves of variable stars.
\newblock {\em Astrophysics and Space Science}, 360(1):5.

\bibitem[Honaker et~al., 2011]{honaker_amelia_2011}
Honaker, J., King, G., and Blackwell, M. (2011).
\newblock Amelia {II}: {A} {Program} for {Missing} {Data}.
\newblock {\em Journal of Statistical Software}, 45(1):1--47.
\newblock Number: 1.

\bibitem[Jakobsen et~al., 2017]{jakobsen_when_2017}
Jakobsen, J.~C., Gluud, C., Wetterslev, J., and Winkel, P. (2017).
\newblock When and how should multiple imputation be used for handling missing
  data in randomised clinical trials – a practical guide with flowcharts.
\newblock {\em BMC Medical Research Methodology}, 17(1):162.

\bibitem[Jochen et~al., 2013]{jochen_multiple_2013}
Jochen, H., Max, H., Tamara, B., and Wilfried, L. (2013).
\newblock Multiple {Imputation} of {Missing} {Data}: {A} {Simulation} {Study}
  on a {Binary} {Response}.
\newblock {\em Open Journal of Statistics}, 2013.
\newblock Publisher: Scientific Research Publishing.

\bibitem[King et~al., 2001]{king_analyzing_2001}
King, G., Honaker, J., Joseph, A., and Scheve, K. (2001).
\newblock Analyzing {Incomplete} {Political} {Science} {Data}: {An}
  {Alternative} {Algorithm} for {Multiple} {Imputation}.
\newblock {\em American Political Science Review}, 95(1):49--69.
\newblock Publisher: Cambridge University Press.

\bibitem[Knol et~al., 2010]{knol_unpredictable_2010}
Knol, M.~J., Janssen, K. J.~M., Donders, A. R.~T., Egberts, A. C.~G., Heerdink,
  E.~R., Grobbee, D.~E., Moons, K. G.~M., and Geerlings, M.~I. (2010).
\newblock Unpredictable bias when using the missing indicator method or
  complete case analysis for missing confounder values: an empirical example.
\newblock {\em Journal of Clinical Epidemiology}, 63(7):728--736.

\bibitem[Leite and Beretvas, 2010]{leite_performance_2010}
Leite, W. and Beretvas, S. (2010).
\newblock The {Performance} of {Multiple} {Imputation} for {Likert}-type
  {Items} with {Missing} {Data}.
\newblock {\em Journal of Modern Applied Statistical Methods}, 9(1).

\bibitem[Little et~al., 2017]{little_conditions_2017}
Little, R.~J., Rubin, D.~B., and Zangeneh, S.~Z. (2017).
\newblock Conditions for {Ignoring} the {Missing}-{Data} {Mechanism} in
  {Likelihood} {Inferences} for {Parameter} {Subsets}.
\newblock {\em Journal of the American Statistical Association},
  112(517):314--320.
\newblock Publisher: Taylor \& Francis \_eprint:
  https://doi.org/10.1080/01621459.2015.1136826.

\bibitem[Little and Rubin, 2019]{little_statistical_2019}
Little, R. J.~A. and Rubin, D.~B. (2019).
\newblock {\em Statistical {Analysis} with {Missing} {Data}}.
\newblock John Wiley \& Sons.
\newblock Google-Books-ID: BemMDwAAQBAJ.

\bibitem[Lupu, 2013]{lupu_best_2013}
Lupu, Y. (2013).
\newblock Best {Evidence}: {The} {Role} of {Information} in {Domestic}
  {Judicial} {Enforcement} of {International} {Human} {Rights} {Agreements}.
\newblock {\em International Organization}, 67(3):469--503.
\newblock Publisher: Cambridge University Press.

\bibitem[Lupu and Wallace, 2019]{lupu_violence_2019}
Lupu, Y. and Wallace, G. P.~R. (2019).
\newblock Violence, {Nonviolence}, and the {Effects} of {International} {Human}
  {Rights} {Law}.
\newblock {\em American Journal of Political Science}, 63(2):411--426.
\newblock \_eprint: https://onlinelibrary.wiley.com/doi/pdf/10.1111/ajps.12416.

\bibitem[Madley-Dowd et~al., 2019]{madley-dowd_proportion_2019}
Madley-Dowd, P., Hughes, R., Tilling, K., and Heron, J. (2019).
\newblock The proportion of missing data should not be used to guide decisions
  on multiple imputation.
\newblock {\em Journal of Clinical Epidemiology}, 110:63--73.

\bibitem[Mishra and Khare, 2014]{mishra_comparative_2014}
Mishra, S. and Khare, D. (2014).
\newblock On comparative performance of multiple imputation methods for
  moderate to large proportions of missing data in clinical trials: a
  simulation study.
\newblock {\em Journal of Medical Statistics and Informatics}, 2(1):9.
\newblock Number: 1 Publisher: Herbert Publications Section: Original Research.

\bibitem[Murray, 2018]{murray_multiple_2018}
Murray, J.~S. (2018).
\newblock Multiple {Imputation}: {A} {Review} of {Practical} and {Theoretical}
  {Findings}.
\newblock {\em Statistical Science}, 33(2):142--159.
\newblock Publisher: Institute of Mathematical Statistics.

\bibitem[Nielsen, 1997]{nielsen_inference_1997}
Nielsen, S.~F. (1997).
\newblock Inference and {Missing} {Data}: {Asymptotic} {Results}.
\newblock {\em Scandinavian Journal of Statistics}, 24(2):261--274.
\newblock Publisher: [Board of the Foundation of the Scandinavian Journal of
  Statistics, Wiley].

\bibitem[Nielsen, 2003]{nielsen_proper_2003}
Nielsen, S.~F. (2003).
\newblock Proper and {Improper} {Multiple} {Imputation}.
\newblock {\em International Statistical Review}, 71(3):593--607.
\newblock \_eprint:
  https://onlinelibrary.wiley.com/doi/pdf/10.1111/j.1751-5823.2003.tb00214.x.

\bibitem[Nielsen and Mørup, 2014]{nielsen_non-negative_2014}
Nielsen, S. F.~V. and Mørup, M. (2014).
\newblock Non-negative {Tensor} {Factorization} with missing data for the
  modeling of gene expressions in the {Human} {Brain}.
\newblock In {\em 2014 {IEEE} {International} {Workshop} on {Machine}
  {Learning} for {Signal} {Processing} ({MLSP})}, pages 1--6.
\newblock ISSN: 2378-928X.

\bibitem[Schafer, 1997]{schafer_analysis_1997}
Schafer, J.~L. (1997).
\newblock {\em Analysis of {Incomplete} {Multivariate} {Data}}.
\newblock Chapman and Hall/CRC, New York.

\bibitem[Schafer, 2013]{schafer_norm_2013}
Schafer, P. t. R. b. A. A. N. O. b. J.~L. (2013).
\newblock norm: {Analysis} of multivariate normal datasets with missing values.

\bibitem[Seaman et~al., 2013]{seaman_what_2013}
Seaman, S., Galati, J., Jackson, D., and Carlin, J. (2013).
\newblock What {Is} {Meant} by “{Missing} at {Random}”?
\newblock {\em Statistical Science}, 28(2):257--268.
\newblock Publisher: Institute of Mathematical Statistics.

\bibitem[Shin et~al., 2017]{shin_maximum_2017}
Shin, T., Davison, M.~L., and Long, J.~D. (2017).
\newblock Maximum likelihood versus multiple imputation for missing data in
  small longitudinal samples with nonnormality.
\newblock {\em Psychological Methods}, 22(3):426--449.
\newblock Place: US Publisher: American Psychological Association.

\bibitem[Sung and Geyer, 2007]{sung_monte_2007}
Sung, Y.~J. and Geyer, C.~J. (2007).
\newblock Monte {Carlo} likelihood inference for missing data models.
\newblock {\em The Annals of Statistics}, 35(3):990--1011.
\newblock Publisher: Institute of Mathematical Statistics.

\bibitem[Takai and Kano, 2013]{takai_asymptotic_2013}
Takai, K. and Kano, Y. (2013).
\newblock Asymptotic {Inference} with {Incomplete} {Data}.
\newblock {\em Communications in Statistics - Theory and Methods},
  42(17):3174--3190.
\newblock Publisher: Taylor \& Francis \_eprint:
  https://doi.org/10.1080/03610926.2011.621577.

\bibitem[Tang and Qin, 2012]{tang_efficient_2012}
Tang, C.~Y. and Qin, Y. (2012).
\newblock An efficient empirical likelihood approach for estimating equations
  with missing data.
\newblock {\em Biometrika}, 99(4):1001--1007.

\bibitem[von Hippel, 2016]{von_hippel_new_2016}
von Hippel, P.~T. (2016).
\newblock New {Confidence} {Intervals} and {Bias} {Comparisons} {Show} {That}
  {Maximum} {Likelihood} {Can} {Beat} {Multiple} {Imputation} in {Small}
  {Samples}.
\newblock {\em Structural Equation Modeling: A Multidisciplinary Journal},
  23(3):422--437.
\newblock Publisher: Routledge \_eprint:
  https://doi.org/10.1080/10705511.2015.1047931.

\bibitem[Wu and Jia, 2013]{wu_new_2013}
Wu, W. and Jia, F. (2013).
\newblock A {New} {Procedure} to {Test} {Mediation} {With} {Missing} {Data}
  {Through} {Nonparametric} {Bootstrapping} and {Multiple} {Imputation}.
\newblock {\em Multivariate Behavioral Research}, 48(5):663--691.
\newblock Publisher: Routledge \_eprint:
  https://doi.org/10.1080/00273171.2013.816235.

\end{thebibliography}

\end{document}